\documentclass[12pt]{article}
\usepackage{tikz}
\usetikzlibrary{arrows}
\usepackage[framemethod=tikz]{mdframed}
\usepackage{wrapfig}
\usepackage{amsmath,amssymb}
\usepackage{type1cm}
\usepackage{amsthm}
\usepackage{mathrsfs}
\usepackage{mathtools}
\usepackage{enumerate}
\usepackage[all]{xy} 
\usepackage[vcentermath,enableskew]{youngtab}
\usepackage{ytableau}
\usepackage{diagbox}
\AtBeginDocument{%
   \def\MR#1{}
}

\DeclareMathOperator{\GL}{GL}

\DeclareMathOperator{\GU}{GU}
\DeclareMathOperator{\GSp}{GSp}
\DeclareMathOperator{\ch}{ch}
\DeclareMathOperator{\diag}{diag}

\DeclareMathOperator{\Ad}{Ad}
\DeclareMathOperator{\Adm}{Adm}

\DeclareMathOperator{\pr}{pr}

\DeclareMathOperator{\pfn}{pfn}

\DeclareMathOperator{\supp}{supp}

\DeclareMathOperator{\LP}{LP}

\newcommand\F{\mathbb{F}}

\newcommand\aFq{\overline{\mathbb F}_q}

\newcommand\cB{\mathcal B}
\newcommand\cO{\mathcal O}

\newcommand\Gm{\mathbb G_m}

\newcommand\Q{\mathbb Q}

\newcommand\A{\mathbb A}

\newcommand\G{\mathbb G}
\newcommand\J{\mathbb J}

\newcommand\Z{\mathbb Z}
\newcommand\tW{\tilde W}
\newcommand\tw{\tilde w}

\newcommand\SW{{^S\tilde W}}
\newcommand\SAdm{{^S\mathrm{Adm}}}

\newcommand\tS{\tilde S}
\newcommand\inv{\mathrm{inv}}

\newcommand\vp{\varpi}
\newcommand\Y{X_*(T)}

\newcommand\la{\langle}
\newcommand\ra{\rangle}
\newcommand\pc{\preceq}

\theoremstyle{definition}
\newtheorem{theo}{Theorem}[section]
\newtheorem{prop}[theo]{Proposition}

\newtheorem{lemm}[theo]{Lemma}

\newtheorem{rema}[theo]{Remark}

\newtheorem{thm}{Theorem}[section]

\begin{document}
\title{On the supersingular locus of the Siegel modular variety of genus 3 or 4}
\author{Ryosuke Shimada\footnote{Department of Mathematics/Hakubi center, Kyoto University,
Yoshida-honmachi, Sakyo-ku, Kyoto 606-8501, Japan
\newline \hspace{0.55cm} email: \texttt{rshimada@math.kyoto-u.ac.jp}}\ and\ Teppei Takamatsu\footnote{Department of Mathematics/Hakubi center, Kyoto University,
Yoshida-honmachi, Sakyo-ku, Kyoto 606-8501, Japan
\newline \hspace{0.55cm} email:\texttt{teppeitakamatsu.math@gmail.com}
\newline \hspace{0.55cm} 2010 Mathematics Subject Classification. Primary: 14G35, 20G25.}}
\date{}
\maketitle

\begin{abstract}
We study the supersingular locus of the Siegel modular variety of genus 3 or 4.
More concretely, we decompose the supersingular locus into a disjoint union of the product of a Deligne-Lusztig variety of Coxeter type and a finite-dimensional affine space after taking perfection.
\end{abstract}

\section{Introduction}
\label{introduction}
Shimura varieties have been used, with great success, towards applications in number theory.
Many of these applications are based on the study of integral models and their reductions.
It is known that in some cases, the supersingular (or basic) locus of the reduction of a Shimura variety admits a simple description.
For example, Vollaard-Wedhorn \cite{VW11} described the supersingular locus of the Shimura variety of $\GU(1,n-1)$ at an inert prime as a union of classical Deligne-Lusztig varieties.
Simple descriptions such as \cite{VW11} have been applied towards the Kudla-Rapoport program \cite{KR99}, \cite{KR11}, Zhang's Arithmetic Fundamental Lemma \cite{Zhang12} and the Tate conjecture for certain Shimura varieties \cite{TX19}, \cite{HTX17}. 

Motivated by \cite{VW11}, G\"{o}rtz-He \cite{GH15} classified the cases where the supersingular locus is naturally a union of classical Deligne-Lusztig varieties of Coxeter type.
These are called the cases of Coxeter type.
The index set of the corresponding stratification and the closure relations between strata can be described in terms of the Bruhat-Tits building of a certain inner form of the underlying group.
Thus this stratification is called the {\it Bruhat-Tits stratification}.
While for the Siegel case for $\GSp_{2n}$, which we are interested in, such a description is possible only when $n\le 2$.
The case $n=2$ was studied intensively by Katsura-Oort \cite{KO87} and Kaiser \cite{Kaiser97}.
The results were later applied to \cite{KR99} by Kudla-Rapoport.
Richartz \cite{Richartz98} studied the case $n=3$ in analogy to the case $n=2$.
In this paper, we study the case $n=3$ or $4$ using more advanced techniques.
The same techniques are also used in \cite{Shimada6} and \cite{Trentin23} to describe the geometry of the basic locus of other Shimura varieties.

The study of the perfection of the supersingular locus is essentially reduced to a study of an affine Deligne-Lusztig variety via the Rapoport-Zink uniformization (see \cite[\S4.1]{Wang21} for example).
Thus the objective of this paper is to find an explicit description of the affine Deligne-Lusztig variety related to the Siegel modular variety of genus 3 or 4 with hyperspecial level structure.

Let $F$ be a non-archimedean local field with finite residue field $\F_q$ of prime characteristic $p$, and let $L$ be the completion of the maximal unramified extension of $F$.
Let $\sigma$ denote the Frobenius automorphism of $L/F$.
Further, we write $\cO$ (resp.\ $\cO_F$) for the valuation ring of $L$ (resp.\ $F$).
Finally, we denote by $\vp$ a uniformizer of $F$ (and $L$) and by $v_L$ the valuation of $L$ such that $v_L(\vp)=1$.

Let $\mu$ be the cocharacter of $\GSp_{2n}$ corresponding to $z\mapsto\diag(z,\ldots,z,1,\ldots,1)$ in which $z$ and $1$ are repeated $n$ times.
Set $\tau={\begin{pmatrix}
0 & \vp^{(1,\ldots,1)} \\
1_n & 0\\
\end{pmatrix}}\in \GSp_{2n}(L)$.
Then its $\sigma$-conjugacy class in $\GSp_{2n}(L)$ is the basic class (cf.\ \cite[\S 1.2]{Wang21}).
Let $X_\mu(\tau)$ denote the affine Deligne-Lusztig variety for $\mu$ and $\tau$.
Let us denote by $\J$ the $\sigma$-centralizer of $\tau$ in $\GSp_{2n}$.
Then $\J$ is an inner form of $\GSp_{2n}$.

\begin{thm}[Theorem \ref{main theo}]
\label{main thm}
Let $n=3$.
The variety $X_\mu(\tau)$ is universally homeomorphic to a union of the product of a Deligne-Lusztig variety of Coxeter type and a finite-dimensional affine space.
The index set of this stratification and the closure relations between strata can be described in terms of the rational Bruhat-Tits building of $\J$.
Moreover, this stratification coincides with the $\J$-stratification.
\end{thm}
The $\J$-stratification, introduced by Chen-Viehmann \cite{CV18}, is a stratification which coincides with the Bruhat-Tits stratification in the case of Coxeter type (see \S\ref{J-str}).
Thus Theorem \ref{main thm} tells us that there exists a natural generalization of the Bruhat-Tits stratification in the case of genus 3.
We call these cases ``positive Coxeter type''.
The cases of positive Coxeter type for $\GL_n$ were studied by the first author in \cite{Shimada5}.

In the case $n=4$, there also exists an analogous stratification.
However, this case is not of positive Coxeter type (cf.\ Remark \ref{not positive}) and we cannot expect a nice property such as the closure relation in Theorem \ref{main thm}.
\begin{thm}[Theorem \ref{main theo2}]
\label{main thm2}
Let $n=4$.
The variety $X_\mu(\tau)$ is universally homeomorphic to a union of the product of a Deligne-Lusztig variety and a finite-dimensional affine space.
The index set of this stratification can be described in terms of the rational Bruhat-Tits building of $\J$.
\end{thm}

For $n\geq 5$, there is no simple description of $X_\mu(\tau)$ such as Theorem \ref{main thm} and Theorem \ref{main thm2} (cf.\ \S5).

The strategy of the proof is as follows:
We first recall the Ekedahl-Oort stratification of $X_\mu(\tau)$ (see \S\ref{ADLV}).
In fact, the stratifications in the main results are refinement of this stratification.
By the non-emptiness criterion of the Iwahori-level affine Deligne-Lusztig variety (see \S\ref{LP}), we may and do eliminate empty Ekedahl-Oort strata.
Then we decompose the remaining non-empty Ekedahl-Oort strata combining the case with spherical $\sigma$-support (Proposition \ref{spherical}) and the Deligne-Lusztig method (see \S\ref{DL method}).
In the case $n=3$, we prove that the resulting stratification coincides with the $\J$-stratification in a similar way to the case of Coxeter type.
The key to the proof of the closure relation is the equidimensionality of $X_\mu(\tau)$, which was established in \cite{HV12} and \cite{Takaya22}.

In a future work, we want to pin down the precise scheme-theoretic structure of the above results, not just up to perfection.
Such results would contribute to an intersection problem involving certain Shimura varieties.

\textbf{Acknowledgments:}
The authors wish to express their gratitude to Naoki Imai for valuable discussion.
The authors would also like to thank Wei Zhang for suggesting potential applications.
Moreover, the authors thank the referee for valuable suggestions.
The first author was supported by the WINGS-FMSP program at the Graduate School of Mathematical Sciences, the University of Tokyo. 
The first author was also supported by JSPS KAKENHI Grant number JP21J22427.
The second author was supported by JSPS KAKENHI Grant number JP22KJ1780.

\section{Preliminaries}
\label{preliminaries}
Keep the notation in \S\ref{introduction}.
From now on, we sometimes drop the adjective ``perfect'' for notational convenience, although we need to work with perfect schemes in most statements in the mixed characteristic setting.
Also, $\cong$ always means a universal homeomorphism.
\subsection{Notation}
\label{notation}
Let $G$ be a split connected reductive group over $F$ and let $T$ be a split maximal torus of it.
Let $B$ be a Borel subgroup of $G$ containing $T$. 
Let $\Phi=\Phi(G,T)$ denote the set of roots of $T$ in $G$.
We denote by $\Phi_+$ (resp.\ $\Phi_-$) the set of positive (resp.\ negative) roots distinguished by $B$.
Let $\Delta$ be the set of simple roots.
Let $X_*(T)$ be the set of cocharacters, and let $X_*(T)_+$ be the set of dominant cocharacters.
For $\mu,\mu'\in X_*(T)$ (resp.\ $X_*(T)_{\Q}$), we write $\mu'\pc \mu$ if $\mu-\mu'$ is a non-negative integral (resp.\ rational) linear combination of positive coroots.
For a cocharacter $\mu\in X_*(T)$, let $\vp^{\mu}$ be the image of $\vp\in \mathbb G_m(F)$ under the homomorphism $\mu\colon\mathbb G_m\rightarrow T$.

Let $B(G)$ denote the set of $\sigma$-conjugacy classes of $G(L)$. 
Thanks to Kottwitz \cite{Kottwitz85}, a $\sigma$-conjugacy class $[b]\in B(G)$ is uniquely determined by two invariants: the Kottwitz point $\kappa(b)\in \pi_1(G)$ and the Newton point $\nu_b\in X_*(T)_{\Q,+}$.
Set $B(G,\mu)=\{[b]\in B(G)\mid \kappa(b)=\kappa(\vp^\mu), \nu_b\pc \mu\}$.

The Iwahori-Weyl group $\tW$ is defined as the quotient $N_{G(L)}T(L)/T(\cO)$.
This can be identified with the semi-direct product $W_0\ltimes X_{*}(T)$, where $W_0$ is the finite Weyl group of $G$.
We denote the projection $\tW\rightarrow W_0$ by $p$.
We have a length function $\ell\colon \tW\rightarrow \Z_{\geq 0}$ given as
$$\ell(w_0\vp^{\lambda})=\sum_{\alpha\in \Phi_+, w_0\alpha\in \Phi_-}|\langle \alpha, \lambda\rangle+1|+\sum_{\alpha\in \Phi_+, w_0\alpha\in \Phi_+}|\langle \alpha, \lambda\rangle|,$$
where $w_0\in W_0$ and $\lambda\in \Y$.
For $w,w'\in \tW$ and $s\in \tS$, we write $w\xrightarrow{s} w'$ if $w'=sws$ and $\ell(w')\le \ell(w)$.

Let $S\subset W_0$ denote the subset of simple reflections, and let $\tS\subset \tW$ denote the subset of simple affine reflections.
We often identify $\Delta$ and $S$.
The affine Weyl group $W_a$ is the subgroup of $\tW$ generated by $\tS$.
Then we can write the Iwahori-Weyl group as a semi-direct product $\tW=W_a\rtimes \Omega$, where $\Omega\subset \tW$ is the subgroup of length $0$ elements.
Moreover, $(W_a, \tS)$ is a Coxeter system.
We denote by $\le$ the Bruhat order on $\tW$ (see \cite[Subsection 1.8]{KR00} for example).
For any $J\subseteq \tS$, let $^J\tW$ be the set of minimal length representatives for the cosets in $W_J\backslash \tW$, where $W_J$ denotes the subgroup of $\tW$ generated by $J$.

For $w\in W_a$, we denote by $\supp(w)\subseteq \tS$ the set of simple affine reflections occurring in every (equivalently, some) reduced expression of $w$.
Note that $\tau\in \Omega$ acts on $\tS$ by conjugation.
We define the $\sigma$-support $\supp_\sigma(w\tau)$ of $w\tau$ as the smallest $\tau$-stable subset of $\tS$ containing $\supp(w)$ (we should consider $\tau\sigma$-action in general, but the action of $\sigma$ is trivial in our case because $G$ is split over $F$).
We call an element $w\tau\in W_a\tau$ a $\sigma$-Coxeter element if exactly one simple reflection from each $\tau$-orbit on $\supp_\sigma(w\tau)$ occurs in every (equivalently, any) reduced expression of $w$.

For $\tau\in \Omega$, let $\Sigma$ be an orbit of $\tau$ on $\tS$ and suppose that $W_\Sigma$ is finite.
We denote by $s_\Sigma$ the unique longest element of $W_\Sigma$.
Then $s_\Sigma$ is fixed by $\tau$.
The fixed point group $W_a^{\tau}\coloneqq \{w\in W_a\mid \tau w\tau^{-1}=w\}$ is the Weyl group whose simple reflections are the elements $s_\Sigma$ such that $\Sigma$ is a $\tau$-orbit on $\tS$ with $W_\Sigma$ finite.
For any reduced decomposition $w=s_{\Sigma_1}\cdots s_{\Sigma_{r}}$ as an element of $W^{\tau}_a$, we have $$\ell(w)=\ell(s_{\Sigma_1})+\cdots +\ell(s_{\Sigma_r}).$$
See \cite{Steinberg68} or \cite[\S2]{KR00} for these facts.

Set $K=G(\cO)$.
For $\alpha\in \Phi$, let $U_\alpha\subseteq G$ denote the corresponding root subgroup.
We also set $$I=T(\cO)\prod_{\alpha\in \Phi_+}U_{\alpha}(\vp\cO)\prod_{\beta\in \Phi_-}U_{\beta}(\cO)\subseteq K,$$
which is called the standard Iwahori subgroup associated with the triple $T\subset B\subset G$.

In the case $G=\GL_n$, we will use the following description.
Let $T$ be the torus of diagonal matrices, and we choose the subgroup of upper triangular matrices $B$ as a Borel subgroup. 
Let $\chi_{ij}$ be the character $T\rightarrow \Gm$ defined by $\mathrm{diag}(t_1,t_2,\ldots, t_n)\mapsto t_i{t_j}^{-1}$.
Then we have $\Phi=\{\chi_{ij}\mid i\neq j\}$, $\Phi_+=\{\chi_{ij}\mid i< j\}$, $\Phi_-=\{\chi_{ij}\mid i> j\}$ and $\Delta=\{\chi_{i,i+1}\mid 1\le i <n\}$.
Through the natural isomorphism $X_*(T)\cong \Z^n$, ${X_*(T)}_+$ can be identified with the set $\{(m_1,\cdots, m_n)\in \Z^n\mid m_1\geq \cdots \geq m_n\}$.
The finite Weyl group is the symmetric group of degree $n$.
Then $S=\{(1\ 2), (2\ 3),\ldots, (n-1\ n)\}$ and $\tS=S\cup\{\vp^{\chi_{1,n}^{\vee}}(1\ n)\}$.
The Iwahori subgroup $I\subset K$ is the inverse image of the lower triangular matrices under the projection $K\rightarrow G(\aFq)$ induced by $\vp\mapsto 0$.

Let us denote by $\GSp_{2n}\subset \GL_{2n}$ the group of symplectic similitudes of dimension $2n$ as in \cite[\S 2.3]{GC10}.
In the case $G=\GSp_{2n}$, we will use the following description.
Let $T$ (resp.\ $B$) be the intersection of the torus (resp.\ Borel subgroup) of $\GL_{2n}$ as above with $\GSp_{2n}$.
See \cite[\S 8]{GC12} for the description of the corresponding roots.
The cocharacter group $X_*(T)$ can be identified with the set $\{(m_1,\cdots, m_{2n})\in \Z^{2n}\mid m_1+m_{2n}=m_2+m_{2n-1}=\cdots=m_n+m_{n+1}\}$.
Set $s_1=(1\ 2)(2n-1\ 2n), s_2=(2\ 3)(2n-2\ 2n-1),\ldots, s_{n-1}=(n-1\ n)(n+1\ n+2), s_n=(n\ n+1)$.
Then $S=\{s_1,s_2,\ldots,s_n\}$ and the finite Weyl group is the subgroup of the symmetric group of degree $2n$ generated by $S$.
Set $s_0=\vp^{\chi_{1,2n}^{\vee}}(1\ 2n)$.
Then $\tS=S\cup\{s_0\}$.
The standard Iwahori subgroup is the intersection of the standard Iwahori subgroup of $\GL_{2n}$ as above with $\GSp_{2n}$.

\subsection{Affine Deligne-Lusztig Varieties}
\label{ADLV}
For $w\in \tW$ and $b\in G(L)$, the affine Deligne-Lusztig variety $X_w(b)$ in the affine flag variety $G(L)/I$ is defined as
$$X_w(b)=\{xI\in G(L)/I\mid x^{-1}b\sigma(x)\in IwI\}.$$
For $\mu\in \Y_+$ and $b\in G(L)$, the affine Deligne-Lusztig variety $X_{\mu}(b)$ in the affine Grassmannian $G(L)/K$ is defined as
$$X_{\mu}(b)=\{xK\in G(L)/K\mid x^{-1}b\sigma(x)\in K\vp^{\mu}K\}.$$
Depending on whether $\ch F>0$ or $0$, both the affine flag variety and the affine Grassmannian are ind-schemes or ind-perfect schemes (see \cite{PR08} or \cite{Zhu17},\cite{BS17} respectively).
Then the affine Deligne-Lusztig varieties are locally closed subvarieties of them equipped with the reduced scheme structure.
Also, the affine Deligne-Lusztig varieties carry a natural action (by left multiplication) by the $\sigma$-centralizer of $b$
$$\J=\J_b=\{g\in G(L)\mid g^{-1}b\sigma(g)=b\}.$$
(Since $b$ is usually fixed in the discussion, we mostly omit it from the notation.)

The admissible subset of $\tW$ associated with $\mu$ is defined as
$$\Adm(\mu)=\{w\in \tW\mid w\le \vp^{w_0\mu}\ \text{for some}\ w_0\in W_0\}.$$
Set $\SAdm(\mu)=\Adm(\mu)\cap \SW$.
Assume that $\mu$ is minuscule.
Then, by \cite[Theorem 3.2.1]{GH15} (see also \cite[\S2.5]{GHR20}), we have
$$X_{\mu}(b)=\bigsqcup_{w\in\SAdm(\mu)}\pi(X_w(b)),$$
where $\pi\colon G(L)/I\rightarrow G(L)/K$ is the projection.
This is the so-called Ekedahl-Oort stratification.
In the sequel, we set $\SAdm(\mu)_0\coloneqq\{w\in \SAdm(\mu)\mid X_w(\tau_\mu)\neq \emptyset\}$, where $\tau_\mu\in \Omega$ such that $[\tau_\mu]\in B(G,\mu)$ is the unique basic element.

Set $S_w=\max\{S'\subseteq S\mid \Ad(w)(S')=S'\}$.
The following proposition is the key to the explicit description of the affine Deligne-Lusztig varieties.
\begin{prop}
\label{spherical}
Let $\tau\in \Omega$.
Let $w\in W_a\tau$ such that $W_{\supp_\sigma(w)}$ is finite.
Then $$X_w(\tau)=\bigsqcup_{j\in \J_\tau/\J_\tau\cap P_{\supp_\sigma(w)}} jY(w),$$
where $P_{\supp_\sigma(w)}$ is the standard parahoric subgroup corresponding to $\supp_\sigma(w)$ and $Y(w)=\{gI\in P_{\supp_\sigma(w)}/I\mid g^{-1}\tau \sigma(g)\in IwI\}$ is a classical Deligne-Lusztig variety in the finite-dimensional flag variety $P_{\supp_\sigma(w)}/I$.

Let $w\in \SW \cap W_a\tau$ be a $\sigma$-Coxeter element in the finite Weyl group $W_{\supp_\sigma(w)}$.
Then $$\pi(X_w(\tau))=\bigsqcup_{j\in \J_\tau/\J_\tau\cap P_{\supp_\sigma(w)\cup S_w}} j\pi(Y(w)),$$
where $P_{\supp_\sigma(w)\cup S_w}$ is the standard parahoric subgroup corresponding to $\supp_\sigma(w)\cup S_w$.
Moreover $\pi$ induces $Y(w)\cong \pi(Y(w))$.
\end{prop}
\begin{proof}
See \cite[Proposition 2.2.1 and Corollary 4.6.2]{GH15}.
\end{proof}

\subsection{The $\J$-stratification}
\label{J-str}
For any $g,h\in G(L)$, let $\inv(g,h)$ (resp.\ $\inv_K(g,h)$) denote the relative position, i.e., the unique element in $\tW$ (resp.\ $\Y_+$) such that $g^{-1}h\in I\inv(g,h)I$ (resp.\ $K\vp^{\inv_K(g,h)} K$).
By definition, two elements $gI,hI\in G(L)/I$ (resp.\ $gK,hK\in G(L)/K$) lie in the same $\J$-stratum if and only if for all $j\in \J$, $\inv(j,g)=\inv(j,h)$ (resp.\ $\inv_K(j,g)=\inv_K(j,h)$).
Clearly, this does not depend on the choice of $g,h$.
By \cite[Theorem 2.10]{Gortz19}, the $\J$-strata are locally closed in $G(L)/I$ (resp.\ $G(L)/K$).
By intersecting each $\J$-stratum with affine Deligne-Lusztig varieties, we obtain the $\J$-stratification of them.

As explained in \cite[Remark 2.1]{CV18}, the $\J$-stratification heavily depends on the choice of $b$ in its $\sigma$-conjugacy class.
Thus we need to fix a specific representative to compare the $\J$-stratification on $X_{\mu}(b)$ to another stratification.
It is pointed out in loc. cit that if $[b]$ is a basic class in $B(G, \mu)$, then a reasonable choice of $b$ is the unique length $0$ element $\tau_{\mu}$.
Also, for any $w\in \tW$, the $\J_{\dot w}$-stratification is independent of the choice of lift $\dot w$ in $G(L)$ (cf.\  \cite[Lemma 2.5]{Gortz19}).
In the rest of this subsection, we fix $b=\tau_\mu$ and hence $\J=\J_{\tau_\mu}$.
In the case of $G=\GSp_{2n}$, we set $\J^0=\{j\in \J\mid \kappa(j)=0\}$ (note that $\pi_1(\GSp_{2n})\cong \Z$).

Assume that $\mu$ is minuscule.
We say that $(G,\mu)$ is of Coxeter type if $$\SAdm(\mu)_0=\{w\in \SAdm(\mu)\mid \text{$W_{\supp_\sigma(w)}$ is finite, and $w$ is $\sigma$-Coxeter in $W_{\supp_\sigma(w)}$}\}.$$
If $(G,\mu)$ is of Coxeter type, then for each $w\in \SAdm(\mu)_0$, we have
$$\pi(X_w(\tau_\mu))=\bigsqcup_{j\in \J/\J\cap P_{\supp_\sigma(w)\cup S_w}}j \pi(Y(w))\quad\text{and}\quad Y(w)\cong \pi(Y(w))$$
by Proposition \ref{spherical}.
Thus if $(G,\mu)$ is of Coxeter type, we obtain the decomposition $X_{\mu}(\tau_\mu)$ as a union of classical Deligne-Lusztig varieties of Coxeter type in a natural way.
We call this stratification the {\it Bruhat-Tits stratification}.
Also, this is a stratification in the strong sense, i.e., the closure of a stratum is a union of strata.
The closure of a stratum $j\pi(Y(w))$ contains a stratum $j'\pi(Y(w'))$ if and only if the following two conditions are both satisfied:
\begin{enumerate}[(1)]
\item $w\geq_{S}w'$, which means by definition that there exists $u\in W_0$ such that $w\geq u^{-1}w'u$.
\item $j(\J\cap P_w)\cap j'(\J\cap P_{w'})\neq \emptyset$.
\end{enumerate}
By \cite[\S 4.7]{He07}, $\geq_{S}$ gives a partial order on $\SW$.
Let $\cB(\J, F)$ denote the rational Bruhat-Tits building of $\J$.
Then (2) above is equivalent to requiring that $\kappa(j)=\kappa(j')$ and that the simplices in $\cB(\J, F)$ corresponding to $j(\J\cap P_w)j^{-1}$ and $j'(\J\cap P_{w'})j'^{-1}$ are neighbors (i.e., there exists an alcove which contains both of them).
In \cite[\S4]{CV18}, Chen-Viehmann conjectured that the Bruhat-Tits stratification coincides with the $\J$-stratification and verified this conjecture in the Siegel case of genus 2 and the Vollaard-Wedhorn case.
In \cite{Gortz19}, G\"ortz proved this conjecture in general.

We go back to the general situation with minuscule $\mu$.
We consider the following conditions (the case of genus 3 in \S\ref{genus 3} satisfies all of them).
\begin{itemize}
\item For $w\in \SAdm(\mu)_0$, $\J$ acts transitively on the set of irreducible components of $X_w(\tau_\mu)$.
\item For $w\in \SAdm(\mu)_0$, there exists a standard parahoric subgroup $P_w\subset G(L)$ and an irreducible component $Y(w)$ of $X_w(\tau_\mu)$ such that $\pi(X_w(\tau_\mu))=\bigsqcup_{j\in \J/\J\cap P_w}j \pi(Y(w))$.
\item $Y(w)\cong \pi(Y(w))$ and each $j \pi(Y(w))$ is a $\J$-stratum of $X_{\mu}(\tau_\mu)$.
\end{itemize}
In this case, we say that the closure relation can be described in terms of $\cB(\J, F)$ if the $\J$-stratification of $X_{\mu}(b)$ is a stratification in the strong sense and $\overline{j\pi(Y(w))}\supseteq j'\pi(Y(w'))$ is equivalent to the following condition:
\begin{enumerate}[($\ast$)]
\item There exist sequences $w=w_0\geq_{S} w_1\geq_{S}\cdots \geq_{S}w_k=w'$ in $\SAdm(\mu)_0$ and $j=j_0, j_1\ldots, j_k=j'$ in $\J$ such that $j_{i-1}(\J\cap P_{w_{i-1}})\cap j_i(\J\cap P_{w_i})\neq \emptyset$ for $1\le i\le k$.
\end{enumerate}
We write $(j,w)\geq (j',w')$ if $(\ast)$ holds.

\subsection{Deligne-Lusztig Reduction Method}
\label{DL method}
The following Deligne-Lusztig reduction method was established in \cite[Corollary 2.5.3]{GH10}.
\begin{prop}
\label{DL method prop}
Let $w\in \tW$ and let $s\in \tS$ be a simple affine reflection.
If $\ch(F)>0$, then the following two statements hold for any $b\in G(L)$.
\begin{enumerate}[(i)]
\item If $\ell(sws)=\ell(w)$, then there exists a $\J_b$-equivariant universal homeomorphism $X_w(b)\rightarrow X_{sws}(b)$.
\item If $\ell(sws)=\ell(w)-2$, then there exists a decomposition $X_w(b)=X_1\sqcup X_2$ such that
\begin{itemize}
\item $X_1$ is open and there exists a $\J_b$-equivariant morphism $X_1\rightarrow X_{sw}(b)$, which is  the composition of a Zariski-locally trivial $\G_m$-bundle and a universal homeomorphism. 
\item $X_2$ is closed and there exists a $\J_b$-equivariant morphism $X_2\rightarrow X_{sws}(b)$, which is the composition of a Zariski-locally trivial $\A^1$-bundle and a universal homeomorphism. 
\end{itemize}
If $\ch(F)=0$, then the above statements still hold by replacing $\A^1$ and $\G_m$ by $\A^{1,\pfn}$ and $\G_m^{\pfn}$ respectively.
\end{enumerate}
\end{prop}

We sketch the construction of maps in the proposition.
Let $gI\in X_w(b)$.
If $\ell(sw)<\ell(w)$ (we can reduce to this case by exchanging $w$ and $sws$), then let $g_1I$ denote the unique element in $G(L)/I$ such that $\inv(g, g_1)=s$ and $\inv(g_1,b\sigma(g))=sw$.
In the case of (ii), the set $X_1$ (resp.\ $X_2$) above consists of the elements $gI\in X_w(b)$ satisfying $\inv(g_1,b\sigma(g_1))=sw$ (resp.\ $sws$).
All of the maps in the proposition are given as the map sending $gI$ to $g_1I$.

\begin{rema}
\label{trivial}
Let $\pr_I\colon G(L)/I\times G(L)/I\rightarrow G(L)/I$ be the projection to the first factor.
We denote by $O(s)\subset G(L)/I\times G(L)/I$ the locally closed subvariety of pairs $(gI,hI)$ such that $\inv(g,h)=s$.
Then the restriction $\pr_I\colon O(s)\rightarrow G(L)/I$ is a Zariski-locally trivial $\A^1$-bundle.
In particular, this is trivial over any Schubert cell $IvI/I, v\in \tW$.
This implies that the morphism $X_1\rightarrow X_{sw}(b)$ (resp.\ $X_2\rightarrow X_{sw\sigma(s)}(b)$) in (ii) is trivial over $X_{sw}(b)\cap IvI/I$ (resp.\ $X_{sw\sigma(s)}(b)\cap IvI/I$).
\end{rema}

The following lemma will be used in \S3 and \S4.
\begin{lemm}
\label{closed}
Let $\mu$ be a minuscule cocharacter and let $w\in \SAdm(\mu)_0$.
Then $X_w(\tau_\mu)$ is closed in $\pi^{-1}(\pi(X_w(\tau_\mu)))$.
\end{lemm}
\begin{proof}
It suffices to prove that
$$\pi^{-1}(\pi(X_{w}(\tau_\mu)))\cap (\bigcup_{w'\le w}X_{w'}(\tau_\mu))=X_{w}(\tau_\mu).$$
We may assume that $w$ has a length greater than $0$.
Let $w''\in \SAdm(\mu)$ with $\ell(w'')< \ell(w)$.
Then we have $\pi(X_{w''}(\tau_\mu))\cap \pi(X_{w}(\tau_\mu))=\emptyset$.
By \cite[\S 3.2 (3)]{GH15}, we also have $\pi(X_{w'}(\tau_\mu))=\pi(X_{w''}(\tau_\mu))$ and hence $\pi(X_{w'}(\tau_\mu))\cap \pi(X_w(\tau_\mu))=\emptyset$ for any $w'\in W_{S_{w''}}w''$.
Note that $\SW\cap W_0\vp^\mu W_0=\SAdm(\mu)$ (cf.\ \cite[(2.3.3)]{Macdonald03}).
Then the above equality follows from \cite[Proposition 3.1.1]{GH15} and Proposition \ref{DL method prop}.
\end{proof}

\subsection{Length Positive Elements}
\label{LP}
We denote by $\delta^+$ the indicator function of the set of positive roots, i.e.,
$$\delta^+\colon \Phi\rightarrow \{0,1\},\quad \alpha \mapsto
\begin{cases}
1 & (\alpha\in \Phi_+) \\
0 & (\alpha\in \Phi_-).
\end{cases}$$
Note that any element $w\in \tW$ can be written in a unique way as $w=x\vp^\mu y$ with $\mu$ dominant, $x,y\in W_0$ such that $\vp^\mu y\in \SW$.
We have $p(w)=xy$ and $\ell(w)=\ell(x)+\la\mu, 2\rho\ra-\ell(y)$.
We define the set of {\it length positive} elements by $$\LP(w)=\{v\in W_0\mid \la v\alpha,y^{-1}\mu\ra+\delta^+(v\alpha)-\delta^+(xyv\alpha)\geq 0\  \text{for all $\alpha\in \Phi_+$}\}.$$
Then we always have $y^{-1}\in \LP(w)$.
Indeed $y$ satisfies the condition that
$\la\alpha, \mu\ra\geq \delta^+(-y^{-1}\alpha)\ \text{for all $\alpha\in \Phi_+$}$.
Since $\delta^+(\alpha)+\delta^+(-\alpha)=1$, we have $$\la y^{-1}\alpha, y^{-1}\mu\ra+\delta^+(y^{-1}\alpha)-\delta^+(x\alpha)=\la \alpha,\mu\ra-\delta^+(-y^{-1}\alpha)+\delta^+(-x\alpha)\geq 0.$$
The notion of length positive elements was defined by Schremmer in \cite{Schremmer22}.

The following theorem is a refinement of the non-emptiness criterion in \cite{GHN15}, which is conjectured by Lim in \cite{Lim23} and proved by Schremmer in \cite[Proposition 5]{Schremmer23}.
\begin{theo}
\label{empty}
Assume that the Dynkin diagram of $G$ is connected.
Let $b\in G(L)$ be a basic element with $\kappa(b)=\kappa(w)$.
Then $X_w(b)=\emptyset$ if and only if the following two conditions are both satisfied:
\begin{enumerate}[(i)]
\item $|W_{\supp_\sigma(w)}|$ is not finite.
\item There exists $v\in \LP(w)$ such that $\supp(v^{-1}p(w)v)\subsetneq S$.
\end{enumerate}
\end{theo}

\begin{rema}
If $\kappa(b)\neq\kappa(w)$, then $X_w(b)=\emptyset$.
\end{rema}

For $w\in \tW$, we say that $w$ has positive Coxeter part if there exists $v\in \LP(w)$ such that $v^{-1}p(w)v$ is a partial Coxeter element.
By \cite[Theorem 4.7]{SSY23} (see also \cite[Theorem 1.1]{HNY22}), this condition induces a simple geometric structure.
\begin{theo}
\label{simple}
Assume that $w\in \tW$ has positive Coxeter part and $X_w(b)\neq \emptyset$.
Then $X_w(b)$ has only one $\J_b$-orbit of irreducible components, and each irreducible component is an iterated fibration over a Deligne-Lusztig variety of Coxeter type whose iterated fibers are either $\A^1$ or $\G_m$.
If $b$ is basic, then all fibers in any such an iterated fibration given by Deligne-Lusztig reduction method are $\A^1$.
\end{theo}
In this paper, we use Theorem \ref{simple} only in Remark \ref{not positive} and \S\ref{genus 5}.

For minuscule $\mu\in \Y$, we say that $(G,\mu)$ is of positive Coxeter type if every $w\in \SAdm(\mu)_0$ satisfies one of the following conditions:
\begin{enumerate}[(i)]
\item $w$ is a $\sigma$-Coxeter element with $|W_{\supp_\sigma(w)}|$ finite.
\item $w$ has positive Coxeter part.
\end{enumerate}
By \cite[Theorem 4.12]{SSY23}, this definition coincides with the one in \cite{Shimada5} if $G=\GL_n$.

\section{The Case of Genus 3}
\label{genus 3}
In this section, we set $G=\GSp_6$ and $\mu=(1^{(3)},0^{(3)})\in X_*(T)$.
Recall that $W_0$ is the subgroup of the symmetric group of degree $6$ generated by
$$s_1=(1\ 2)(5\ 6),\quad s_2=(2\ 3)(4\ 5),\quad s_3=(3\ 4).$$
Moreover, the affine Weyl group $W_{a}$ is generated by $s_1, s_2, s_3$ and the simple affine reflection 
\[
s_{0} = \varpi^{(1, 0,0,0,0 ,-1)} (1\ 6) \in W_{a}.
\]
Set $\tau=\tau_\mu=\vp^{\mu}s_3s_2s_1s_3s_2s_3=\vp^\mu(1\ 4)(2\ 5)(3\ 6)\in \Omega$ and $\J=\J_\tau$.
Then $$\tau s_1\tau^{-1}=s_2,\quad\tau s_0\tau^{-1}=s_3.$$
Thus there are two $\tau$-orbit in $\tS$, namely, $\{s_0,s_3\}$ and $\{s_1,s_2\}$.
Then $W_a^{\tau}$ is the Weyl group with two simple reflections $s_0s_3$ and $s_1s_2s_1$.
For $j,j'\in \J^0$, we have $\inv(j,j')\in W_a^{\tau}$.
\begin{prop}
\label{genus 3 empty}
We have 
\begin{align*}
\SAdm(\mu)_0=\{s_0s_1s_0\tau,s_0\tau,\tau\}.
\end{align*}
Moreover, the pair $(G,\mu)$ is of positive Coxeter type.
\end{prop}
\begin{proof}
It is straightforward to check that $\SAdm(\mu)$ is equal to
\begin{align*}
\{\vp^{\mu}, \vp^{\mu} s_3,\vp^{\mu}s_3s_2,\vp^{\mu}s_3s_2s_1,\vp^{\mu}s_3s_2s_3,\vp^{\mu}s_3s_2s_1s_3,\vp^{\mu}s_3s_2s_1s_3s_2,\vp^{\mu}s_3s_2s_1s_3s_2s_3\}.
\end{align*}
If $w=\vp^{\mu}, \vp^{\mu} s_3,\vp^{\mu}s_3s_2$ or $\vp^{\mu}s_3s_2s_3$, then $\supp_\sigma(w)=\tS$ and $\supp_\sigma(p(w))\subsetneq S$.
Thus $X_w(\tau)=\emptyset$ by Theorem \ref{empty}.
It is easy to check that $s_3s_1s_2\in \LP(\vp^\mu s_3s_2s_1s_3)$.
Since $\supp_\sigma(\vp^{\mu}s_3s_2s_1s_3)=\tS$ and $\supp_\sigma((s_3s_1s_2)^{-1}s_3s_2s_1s_3(s_3s_1s_2))\subsetneq S$, we have $X_{\vp^{\mu}s_3s_2s_1s_3}(\tau)=\emptyset$ by Theorem \ref{empty}.
Thus $$\SAdm(\mu)_0\subseteq\{\vp^{\mu}s_3s_2s_1,\vp^{\mu}s_3s_2s_1s_3s_2,\vp^{\mu}s_3s_2s_1s_3s_2s_3\}.$$
For $w=\vp^{\mu}s_3s_2s_1s_3s_2,\vp^{\mu}s_3s_2s_1s_3s_2s_3$, we have $\supp_\sigma(w)\neq \tS$.
Note also that $p(\vp^\mu s_3s_2s_1)=s_3s_2s_1$ is a Coxeter element and hence $\supp(v^{-1}p(w)v)=S$ for any $v\in W_0$.
Therefore 
$\SAdm(\mu)_0=\{\vp^{\mu}s_3s_2s_1,\vp^{\mu}s_3s_2s_1s_3s_2,\vp^{\mu}s_3s_2s_1s_3s_2s_3\}=\{s_0s_1s_0\tau,s_0\tau,\tau\}$.
Since $\tau$ and $s_0\tau$ are $\sigma$-Coxter elements, $(G,\mu)$ is of positive Coxeter type.
This finishes the proof.
\end{proof}

\begin{rema}
The element $s_0\tau$ does not have positive Coxeter part.
Indeed, the cycle type of $p(s_0\tau)=(1\ 4\ 6\ 3)(2\ 5)$ is different from that of any Coxeter element in $W_0$.
\end{rema}

By Proposition \ref{spherical}, we have
$$X_{s_0\tau}(\tau)=\bigsqcup_{j\in\J/\J \cap P_{\{s_0,s_3\}}}jY(s_0\tau)\quad\text{and}\quad X_{\tau}(\tau)=\bigsqcup_{j\in\J/\J \cap I}j\{pt\},$$
where $Y(s_0\tau)$ and $Y(\tau)=\{pt\}$ are classical Deligne-Lusztig varieties as in Proposition \ref{spherical}.
By \cite[Corollary 2.5]{Lusztig76} (see also \cite[Proposition 1.1]{Gortz19}), $Y(s_0\tau)\subset Is_0s_3I/I$.

\begin{lemm}
\label{genus 3 Iwahori}
There exists an irreducible component $Y(s_0s_1s_0\tau)$ of $X_{s_0s_1s_0\tau}(\tau)$ such that
$$X_{s_0s_1s_0\tau}(\tau)=\bigsqcup_{j\in \J/\J\cap P_{\{s_1,s_2\}}}jY(s_0s_1s_0\tau)\quad\text{and}\quad Y(s_0s_1s_0\tau)\cong Y(s_1\tau)\times \A^1,$$
where $Y(s_1\tau)$ is a classical Deligne-Lusztig variety as in Proposition \ref{spherical}.
Moreover each $jY(s_0s_1s_0\tau)$ is contained in a $\J$-stratum in $G(L)/I$.
\end{lemm}
\begin{proof}
We have
$$s_0s_1s_0\tau \xrightarrow{s_0} s_1s_0s_3\tau\xrightarrow{s_3} s_1\tau.$$
Note that $(s_2s_1)s_1s_0\tau(s_2s_1)^{-1}=s_0s_1\tau=\vp^{\mu}s_3s_2s_1s_3$.
Thus $X_{s_1s_0\tau}(\tau)=\emptyset$ by Proposition \ref{DL method prop} (i) and the proof of Proposition \ref{genus 3 empty}.
Let $f\colon X_{s_0s_1s_0\tau}(\tau)\rightarrow X_{s_1\tau}(\tau)$ be the morphism induced by Proposition \ref{DL method prop}.
By Proposition \ref{spherical}, we have
$$X_{s_1\tau}(\tau)=\bigsqcup_{j\in\J/\J \cap P_{\{s_1,s_2\}}}jY(s_1\tau).$$
We set $Y(s_0s_1s_0\tau)=f^{-1}(Y(s_1\tau))$.
By \cite[Corollary 2.5]{Lusztig76} (see also \cite[Proposition 1.1]{Gortz19}), $Y(s_1\tau)\subset Is_1s_2s_1I/I$.
Clearly $s_1s_2s_1s_3s_0$ is a reduced expression.
By Remark \ref{trivial}, we have
$$X_{s_0s_1s_0\tau}(\tau)=\bigsqcup_{j\in \J/\J\cap P_{\{s_1,s_2\}}}jY(s_0s_1s_0\tau)\quad\text{and}\quad Y(s_0s_1s_0\tau)\cong Y(s_1\tau)\times \A^1.$$
Since $Y(s_1\tau)$ is an irreducible component of $X_{s_1\tau}(\tau)$, $Y(s_0s_1s_0\tau)$ is also an irreducible component of $X_{s_0s_1s_0\tau}(\tau)$.

It remains to show that for all $j\in \J$ and $w\in \SAdm(\mu)_0$, the value $\inv(j,-)$ is constant on each $j'Y(w)$.
For this, we argue similarly as \cite[\S 3.3]{Gortz19}.
Clearly we may assume $j'=1$.
For any $j\in \J$, there exists $\tilde j\in \J^0$ such that $\inv(j, \tilde j)\in\Omega$.
Thus we may also assume $j\in \J^0$.
Then by \cite[Corollary 2.5]{Lusztig76} and \cite[Proposition 5.34]{AB08} (see also \cite[Proposition 1.7]{Gortz19}), there exists $gI$ with $g\in \J\cap P_{\{s_1,s_2\}}$ such that for any $y_0I\in Y(s_1\tau)$, we have
\begin{align*}
\inv(j, y_0)=\inv(j,g)s_1s_2s_1(\in W_a)\quad \text{with}\quad \ell(\inv(j, y_0))=\ell(\inv(j,g))+3.
\end{align*}
In particular, $\inv(j, g)$ has a reduced expression as an element of $W_a^{\tau}$ whose rightmost simple reflection is $s_0s_3$.
Let $y\in Y(s_0s_1s_0\tau)$ and set $y_0=f(y)\in Y(s_1\tau)$.
Note that $\inv(y_0,y)=s_0s_3$ (cf.\ the comment right after Proposition \ref{DL method prop}) and
\begin{align*}
&\ell(\inv(j, y_0)s_0s_3)=\ell(\inv(j,g))+3+2.
\end{align*}
Thus $\inv(j, y)=\inv(j, y_0)\inv(y_0,y)=\inv(j,g)s_1s_2s_1s_0s_3$ is independent of $y\in Y(w_k)$.
\end{proof}

Similarly as the proof of Lemma \ref{genus 3 Iwahori}, we can show that $jY(s_0\tau)$ is contained in a $\J$-stratum in $G(L)/I$.
The following theorem is our main result in the case of genus 3:
\begin{theo}
\label{main theo}
We have $\pi(Y(s_0s_1s_0\tau))\cong Y(s_1\tau)\times \A^1$, $\pi(Y(s_0\tau))\cong Y(s_0\tau)$ and
$$X_\mu(\tau)=\bigsqcup_{j\in \J/\J\cap P_{\{s_1,s_2\}}}j\pi(Y(s_0s_1s_0\tau))\sqcup \bigsqcup_{j\in \J/\J\cap P_{\{s_0,s_3\}}}j\pi(Y(s_0\tau))\sqcup \bigsqcup_{j\in\J/\J \cap P_{\{s_1,s_2\}}}j\{pt\}.$$
Moreover, each stratum is a $\J$-stratum in $G(L)/K$, and the closure relation can be described in terms of $\cB(\J,F)$.
\end{theo}
\begin{proof}
By Lemma \ref{genus 3 empty}, we have
$$X_\mu(\tau)=\bigsqcup_{w\in \SAdm(\mu)_0}\pi(X_w(\tau))=\pi(X_{s_0s_1s_0\tau}(\tau))\sqcup \pi(X_{s_0\tau}(\tau))\sqcup \pi(X_{\tau}(\tau)).$$
Note that $S_{s_0\tau}=\emptyset$ and $S_{\tau}=\{s_1,s_2\}$.
By Proposition \ref{spherical}, we also have $\pi(Y(s_0\tau))\cong Y(s_0\tau)$,
$$\pi(X_{s_0\tau}(\tau))=\bigsqcup_{j\in \J/\J\cap P_{\{s_0,s_3\}}}j\pi(Y(s_0\tau))\quad \text{and}\quad \pi(X_{\tau}(\tau))=\bigsqcup_{j\in\J/\J \cap P_{\{s_1,s_2\}}}j\{pt\}.$$
By Lemma \ref{closed}, $X_{s_0s_1s_0\tau}(\tau)$ is closed in $\pi^{-1}(\pi(X_{s_0s_1s_0\tau}(\tau)))$.
By \cite[Lemma 5.4]{HNY22}, the map $X_{s_0s_1s_0\tau}(\tau)\rightarrow \pi(X_{s_0s_1s_0\tau}(\tau))$ induced by $\pi$ is injective.
Since $\pi$ is proper and $X_{s_0s_1s_0\tau}(\tau)$ is closed in $\pi^{-1}(\pi(X_{s_0s_1s_0\tau}(\tau)))$, the map $X_{s_0s_1s_0\tau}(\tau)\rightarrow \pi(X_{s_0s_1s_0\tau}(\tau))$ is also proper.
Thus this map is a universal homeomorphism.
Therefore $\pi(Y(s_0s_1s_0\tau))\cong Y(s_1\tau)\times \A^1$ and
$$\pi(X_{s_0s_1s_0\tau}(\tau))=\bigsqcup_{j\in \J/\J\cap P_{\{s_1,s_2\}}}j\pi(Y(s_0s_1s_0\tau))$$
by Lemma \ref{genus 3 Iwahori}.

We next prove the closure relation, i.e.\ we show that for any $j \in \J$, $\overline{j\pi(Y(w))}\supseteq j'\pi(Y(w'))$ if and only if the condition $(\ast)$ in Subsection \ref{J-str} holds true. We may assume that $j = 1$ and we replace $j'$ with $j$ in the following.
We will prove the following two assertions:
\begin{enumerate}[(1)]
\item If $(1,s_0s_1s_0\tau)\geq (j,s_0\tau)$ (i.e., $(\J\cap P_{\{s_1,s_2\}})\cap j(\J\cap P_{\{s_0,s_3\}})\neq \emptyset$), then $j\pi(Y(s_0\tau))\subset \overline{\pi(Y(s_0s_1s_0\tau))}$.
\item Otherwise, $j\pi(Y(s_0\tau))\cap\overline{\pi(Y(s_0s_1s_0\tau))}=\emptyset$.
\end{enumerate}
Clearly $(1,s_0s_1s_0\tau)\ngeq (j,s_0\tau)$ and $j\pi(Y(s_0\tau))\cap\overline{\pi(Y(s_0s_1s_0\tau))}=\emptyset$ if $j\in \J\setminus \J^0$.
Thus we may and do assume that $j\in \J^0$.
By replacing $j$ by another representative in $j(\J\cap P_{\{s_0,s_3\}})$  if necessary, we may also assume that $\inv(1,j)$ has a reduced expression as an element of $W_a^{\tau}$ whose rightmost simple reflection is $s_1s_2s_1$ unless $\inv(1,j)=1$.

Recall that $Y(s_0\tau)$ is contained in $Is_0s_3I/I$.
Thus there exists $\mu_j\in \Y_+$ such that $j\pi(Y(s_0\tau))\subset K\vp^{\mu_j}K/K$.
Note that $(\J\cap P_{\{s_1,s_2\}})\cap j(\J\cap P_{\{s_0,s_3\}})\neq \emptyset$ is equivalent to $\inv(1,j)=1$ or $s_1s_2s_1$.
Moreover, if $(\J\cap P_{\{s_1,s_2\}})\cap j(\J\cap P_{\{s_0,s_3\}})=\emptyset$, then $s_0$ (and hence $s_3$) belongs to $\supp(\inv(1,j))$.
Combining this with \cite[(2.7.11)]{Macdonald03}, we deduce that if $(\J\cap P_{\{s_1,s_2\}})\cap j(\J\cap P_{\{s_0,s_3\}})=\emptyset$, then $(1,1,0,0,-1,-1)\pc \mu_j$.
By the proof of Lemma \ref{genus 3 Iwahori}, we have $\pi(Y(s_0s_1s_0\tau))\subset \pi(Is_1s_2s_1s_0s_3I/I)\subset K\vp^{(1,0,0,0,0,-1)}K/K$.
These imply (2).

The irreducibility of $\pi(Y(s_0s_1s_0\tau))$ and $\pi(Y(s_0\tau))$ combined with the equidimensionality of $X_\mu(\tau)$ (cf.\ \cite{HV12} and \cite{Takaya22}) implies that there exists $j_0\in \J^0$ such that $j_0\pi(Y(s_0\tau))\subset \overline{\pi(Y(s_0s_1s_0\tau))}$.
This $j_0$ must satisfy $(\J\cap P_{\{s_1,s_2\}})\cap j_0(\J\cap P_{\{s_0,s_3\}})\neq \emptyset$.
Thus by multiplying $j_0\pi(Y(s_0\tau))$ by elements in $\J\cap P_{\{s_1,s_2\}}$, we deduce that if $(\J\cap P_{\{s_1,s_2\}})\cap j(\J\cap P_{\{s_0,s_3\}})\neq \emptyset$, then $j\pi(Y(s_0\tau))\subset \overline{\pi(Y(s_0s_1s_0\tau))}$.
This proves (1).

Similarly, we will prove the following assertions:
\begin{enumerate}[(1)]
\setcounter{enumi}{-1}
\item If $j\notin \J\cap P_{\{s_1,s_2\}}$ (i.e., $j\pi(X_{s_0s_1s_0\tau}(\tau))\neq \pi(X_{s_0s_1s_0\tau}(\tau))$), then $j\pi(X_{s_0s_1s_0\tau}(\tau))\cap \overline{\pi(X_{s_0s_1s_0\tau}(\tau))}=\emptyset$.
\end{enumerate}
\begin{enumerate}[(1)']
\item If $(1,s_0s_1s_0\tau)\geq (j, \tau)$, then $j\pi(Y(\tau))\subset \overline{\pi(Y(s_0s_1s_0\tau))}$.
\item Otherwise, $j\pi(Y(\tau))\cap\overline{\pi(Y(s_0s_1s_0\tau))}=\emptyset$.
\end{enumerate}
The statement of (1)' follows from the above discussion and the proof of \cite[Theorem 7.2.1]{GH15}.
By replacing $j$ by another representative in $j(\J\cap P_{\{s_1,s_2\}})$ if necessary, we may assume that $\inv(1,j)$ has a reduced expression as an element of $W_a^{\tau}$ whose rightmost simple reflection is $s_0s_3$ unless $\inv(1,j)=1$.
Thus if $j\notin \J\cap P_{\{s_1,s_2\}}$, then similarly as above, there exists $(1,1,0,0,-1,-1)\pc \mu_j$ such that $j\pi(X_{s_0s_1s_0\tau}(\tau))\subset K\vp^{\mu_j}K/K$.
This proves (0).
If $(1,s_0s_1s_0\tau)\geq (j, \tau)$, then $\inv(1,j)=1, s_0s_3$ or $s_1s_2s_1s_0s_3$.
Thus (2)' also follows in the same way as above.
The closure relation follows from (0), (1), (2), (1)', (2)' and the proof of \cite[Theorem 7.2.1]{GH15}.

It remains to show that each stratum is a $\J$-stratum.
This is clearly true for $j\{pt\}$.
Here we only show that $\pi(Y(s_0s_1s_0\tau))$ is a $\J$-stratum.
Proofs for remaining strata are similar.
Recall that $\pi(Y(s_0s_1s_0\tau))\subset K\vp^{(1,0,0,0,0,-1)}K/K$ and $j\pi(X_{s_0s_1s_0\tau}(\tau))\subset K\vp^{\mu_j}K/K$ with $(1,1,0,0,-1,-1)\pc \mu_j$.
Thus to show that $\pi(Y(s_0s_1s_0\tau))$ is a $\J$-stratum, we need to show that for $j\in \J^0$ such that $j\pi(Y(s_0\tau))\subset K\vp^{(1,0,0,0,0,-1)}K/K$, $j\pi(Y(s_0\tau))$ is contained in a different $\J$-stratum from $\pi(Y(s_0s_1s_0\tau))$.
Similarly as above, we may assume that $j\in \J\cap P_{\{s_1,s_2\}}$.
Let $j'\in \J^0\setminus (\J\cap P_{\{s_1,s_2\}})$ such that $j(\J\cap P_{\{s_0,s_3\}})\cap j'(\J\cap P_{\{s_1,s_2\}})\neq \emptyset$.
Let $y\in Y(s_0s_1s_0\tau)$ and set $y_0=f(y)\in Y(s_1\tau)$ as in the proof of Lemma \ref{genus 3 Iwahori}.
Then there exists $gI$ with $g\in \J\cap P_{\{s_1,s_2\}}$ such that
\begin{align*}
\inv(j', y_0)=\inv(j',g)s_1s_2s_1(\in W_a)\quad \text{with}\quad \ell(\inv(j', y_0))=\ell(\inv(j',g))+3.
\end{align*}
Thus $\inv(j',Y(s_0s_1s_0\tau))=\inv(j',y)=\inv(j',g)s_1s_2s_1s_0s_3$.
Moreover, by the assumption on $j'$, we have $\inv(j',g)\neq 1$ (i.e., $s_0$ belongs to $\supp(\inv(j',g))$).
This implies that $(1,1,0,0,-1,-1)\pc\inv_K(j', \pi(Y(s_0s_1s_0\tau)))$.
On the other hand, by $j'^{-1}j(\J\cap P_{\{s_0,s_3\}})\cap (\J\cap P_{\{s_1,s_2\}})\neq \emptyset$, we have $\inv_K(j',j\pi(Y(s_0\tau)))\pc (1,0,0,0,0,-1)$.
Therefore $j\pi(Y(s_0\tau))$ is contained in a different $\J$-stratum from $\pi(Y(s_0s_1s_0\tau))$.
This finishes the proof.
\end{proof}

\section{The Case of Genus 4}
\label{genus 4}
In this section, we set $G=\GSp_8$ and $\mu=(1^{(4)},0^{(4)})\in X_*(T)$.
Recall that $W_0$ is the subgroup of the symmetric group of degree $6$ generated by
$$s_1=(1\ 2)(7\ 8),\quad s_2=(2\ 3)(6\ 7),\quad s_3=(3\ 4)(5\ 6)\quad s_4=(4\ 5).$$
Moreover, the affine Weyl group $W_{a}$ is generated by $s_1, s_2, s_3, s_4$ and the simple affine reflection 
\[
s_{0} = \varpi^{(1^{(1)}, 0^{(6)}, (-1)^{(1)})} (1\ 8) \in W_{a}.
\]
Set $\tau=\tau_\mu=\vp^{\mu}s_4s_3s_2s_1s_4s_3s_2s_4s_3s_4=\vp^\mu(1\ 5)(2\ 6)(3\ 7)(4\ 8)\in \Omega$ and $\J=\J_\tau$.
Then $$\tau s_1\tau^{-1}=s_3,\quad \tau s_2\tau^{-1}=s_2,\quad\tau s_4\tau^{-1}=s_0.$$
Thus there are two $\tau$-orbit in $\tS$, namely, $\{s_0,s_4\}$, $\{s_1,s_3\}$ and $\{s_2\}$.
Then $W_a^{\tau}$ is the Weyl group with two simple reflections $s_0s_4$, $s_1s_3$ and $s_2$.
For $j,j'\in \J^0$, we have $\inv(j,j')\in W_a^{\tau}$.

\begin{lemm}
\label{genus 4 empty}
We have 
\begin{align*}
\SAdm(\mu)_0=\{s_0s_1s_2s_0s_1s_0\tau, s_0s_1s_2s_0\tau,s_0s_1s_0\tau,s_0s_1\tau,s_0\tau,\tau\}.
\end{align*}
\end{lemm}
\begin{proof}
It is straightforward to check that $\SAdm(\mu)$ is equal to
\begin{align*}
\{&\vp^{\mu}, \vp^{\mu} s_4,\vp^{\mu}s_4s_3,\vp^{\mu}s_4s_3s_2,\vp^{\mu}s_4s_3s_4,\vp^{\mu}s_4s_3s_4s_2,\vp^{\mu}s_4s_3s_4s_2s_3,\vp^{\mu}s_4s_3s_4s_2s_3s_4,\\
&\vp^\mu s_4s_3s_2s_1,\vp^\mu s_4s_3s_2s_1s_4,\vp^\mu s_4s_3s_2s_1s_4s_3,\vp^\mu s_4s_3s_2s_1s_4s_3s_4,\vp^\mu s_4s_3s_2s_1s_4s_3s_2\\
&\vp^\mu s_4s_3s_2s_1s_4s_3s_2s_4,\vp^\mu s_4s_3s_2s_1s_4s_3s_2s_4s_3,\vp^\mu s_4s_3s_2s_1s_4s_3s_2s_4s_3s_4\}.
\end{align*}
If $w=\vp^{\mu}, \vp^{\mu} s_4,\vp^{\mu}s_4s_3,\vp^{\mu}s_4s_3s_2,\vp^{\mu}s_4s_3s_4,\vp^{\mu}s_4s_3s_4s_2,\vp^{\mu}s_4s_3s_4s_2s_3$ or $\vp^{\mu}s_4s_3s_4s_2s_3s_4$, then $\supp_\sigma(w)=\tS$ and $\supp(p(w))\subsetneq S$.
Thus $X_w(\tau)=\emptyset$ by Theorem \ref{empty}.
It is easy to check that $s_4s_1s_2s_3\in \LP(\vp^\mu s_4s_3s_2s_1s_4)$ and $s_4s_3s_4s_1s_2\in \LP(\vp^\mu s_4s_3s_2s_1s_4s_3s_4)$.
Since $\supp_\sigma(\vp^\mu s_4s_3s_2s_1s_4)=\supp_\sigma(\vp^\mu s_4s_3s_2s_1s_4s_3s_4)=\tS$,
\begin{align*}
&(s_4s_1s_2s_3)^{-1}s_4s_3s_2s_1s_4(s_4s_1s_2s_3)=s_3s_2s_1\quad\text{and}\\
&(s_4s_3s_4s_1s_2)^{-1}s_4s_3s_2s_1s_4s_3s_4(s_4s_3s_4s_1s_2)=s_2s_1s_4,
\end{align*}
both $X_{\vp^\mu s_4s_3s_2s_1s_4}(\tau)$ and $X_{\vp^\mu s_4s_3s_2s_1s_4s_3s_4}(\tau)$ are empty.
Thus 
\begin{align*}
\SAdm(\mu)_0\subseteq\{&\vp^\mu s_4s_3s_2s_1,\vp^\mu s_4s_3s_2s_1s_4s_3,\vp^\mu s_4s_3s_2s_1s_4s_3s_2,\\
&\vp^\mu s_4s_3s_2s_1s_4s_3s_2s_4,\vp^\mu s_4s_3s_2s_1s_4s_3s_2s_4s_3,\vp^\mu s_4s_3s_2s_1s_4s_3s_2s_4s_3s_4\}\\
=\{&s_0s_1s_2s_0s_1s_0\tau, s_0s_1s_2s_0\tau, s_0s_1s_0\tau,s_0s_1\tau,s_0\tau,\tau\}.
\end{align*}
For $w=s_0s_1s_0\tau,s_0s_1\tau,s_0\tau$ and $\tau$, we have $\supp_\sigma(w)\subsetneq \tS$.
We also have $\supp_\sigma(s_1s_2\tau)\subsetneq \tS$ and $$s_0s_1s_2s_0\tau \xrightarrow{s_0} s_1s_2s_0s_4\tau\xrightarrow{s_4} s_1s_2\tau.$$
Since $s_4s_3s_2s_1$ is a Coxeter element, we have $\supp(v^{-1}p(s_0s_1s_2s_0s_1s_0\tau)v)=S$ for any $v\in W_0$.
Thus, by Proposition \ref{DL method prop} and Theorem \ref{empty}, the above inclusion is an equality.
This finishes the proof.
\end{proof}

For $w=s_0s_1s_0\tau,s_0s_1\tau,s_0\tau$ and $\tau$, the description of $X_w(\tau)$ follows from Proposition \ref{spherical}.
In particular, they can be written as disjoint unions of Deligne-Lusztig varieties.
\begin{lemm}
\label{genus 4 Iwahori}
Let $w$ be $s_0s_1s_2s_0s_1s_0\tau$ or $s_0s_1s_2s_0\tau$.
Set $e_w=s_0s_3\tau$ (resp.\ $s_1s_2\tau$) in the former (resp.\ latter) case.
Then there exists an irreducible component $Y(w)$ of $X_{w}(\tau)$ such that
$$X_{w}(\tau)=\bigsqcup_{j\in \J/\J\cap P_{\supp_\sigma(e_w)}}jY(w)\quad\text{and}\quad Y(w)\cong Y(e_w)\times \A^{\frac{\ell(w)}{2}-1},$$
where $Y(e_w)$ is a classical Deligne-Lusztig variety as in Proposition \ref{spherical}.
\end{lemm}
\begin{proof}
Let $w=s_0s_1s_2s_0\tau$.
We have $$s_0s_1s_2s_0\tau \xrightarrow{s_0} s_1s_2s_0s_4\tau\xrightarrow{s_4} e_w.$$
It is straightforward to check that $(s_0s_4)s_1s_2s_0\tau(s_0s_4)^{-1}=s_0s_1s_2\tau=\vp^{\mu}s_4s_3s_2s_1s_4s_3s_4$.
Thus $X_{s_1s_2s_0\tau}(\tau)=\emptyset$ by Proposition \ref{DL method prop} (i) and the proof of Lemma \ref{genus 4 empty}.
Let $f\colon X_{w}(\tau)\rightarrow X_{e_w}(\tau)$ be the morphism induced by Proposition \ref{DL method prop}.
By Proposition \ref{spherical}, we have
$$X_{e_w}(\tau)=\bigsqcup_{j\in\J/\J \cap P_{\supp_\sigma(e_w)}}jY(e_w).$$
We set $Y(w)=f^{-1}(Y(e_w))$.
By \cite[Corollary 2.5]{Lusztig76}, $Y(e_w)\subset Is_1s_2s_3s_2s_1s_2I/I$.
It is easy to check that $s_1s_2s_3s_2s_1s_2s_4s_0$ is a reduced expression.
By Remark \ref{trivial}, we have
$$X_{w}(\tau)=\bigsqcup_{j\in \J/\J\cap P_{\supp_\sigma(e_w)}}jY(w)\quad\text{and}\quad Y(w)\cong Y(e_w)\times \A^1.$$

Let $w=s_0s_1s_2s_0s_1s_0\tau$.
We have 
\begin{align*}
s_0s_1s_2s_0s_1s_0\tau \xrightarrow{s_0} s_1s_2s_0s_1s_0s_4\tau\xrightarrow{s_4} s_1s_2s_0s_1\tau\xrightarrow{s_1}s_2s_0s_1s_3\tau\xrightarrow{s_3}s_3s_2s_0s_3\tau\xrightarrow{s_2} e_w.
\end{align*}
It is straightforward to check that $(s_0s_4)s_1s_2s_0s_1s_0\tau (s_0s_4)^{-1}=s_0s_1s_2s_0s_1\tau=\vp^\mu s_4s_3s_2s_1s_4$ and $(s_0s_4s_1)s_0s_2s_3\tau(s_0s_4s_1)^{-1}=s_0s_1s_2\tau=\vp^\mu s_4s_3s_2s_1s_4s_3s_4$.
Thus by Proposition \ref{DL method prop} (i) and the proof of Lemma \ref{genus 4 empty}, both $X_{s_1s_2s_0s_1s_0\tau}(\tau)$ and $X_{s_0s_2s_3\tau}(\tau)$ are empty.
Let $f\colon X_{w}(\tau)\rightarrow X_{e_w}(\tau)$ be the morphism induced by Proposition \ref{DL method prop}.
By Proposition \ref{spherical}, we have
$$X_{e_w}(\tau)=\bigsqcup_{j\in\J/\J \cap P_{\supp_\sigma(e_w)}}jY(e_w).$$
We set $Y(w)=f^{-1}(Y(e_w))$.
By \cite[Corollary 2.5]{Lusztig76}, $Y(e_w)\subset Is_0s_1s_0s_1s_3s_4s_3s_4I/I$.
It is easy to check that $s_0s_1s_0s_1s_3s_4s_3s_4s_2s_3s_1s_4s_0$ is a reduced expression.
By Remark \ref{trivial}, we have
$$X_{w}(\tau)=\bigsqcup_{j\in \J/\J\cap P_{\supp_\sigma(e_w)}}jY(w)\quad\text{and}\quad Y(w)\cong Y(e_w)\times \A^2.$$

In both cases, $Y(e_w)$ is an irreducible component of $X_{e_w}(\tau)$.
Thus $Y(w)$ is also an irreducible component of $X_w(\tau)$.
This finishes the proof.
\end{proof}

The following theorem is our main result in the case of genus 4:
\begin{theo}
\label{main theo2}
We have 
\begin{align*}
\pi(Y(s_0s_1s_2s_0s_1s_0\tau))\cong Y(s_0s_1\tau)\times \A^2,\quad \pi(Y(s_0s_1s_2s_0\tau))\cong Y(s_1s_2\tau)\times \A^1,\\
\pi(Y(s_0s_1s_0\tau))\cong Y(s_0s_1s_0\tau),\quad \pi(Y(s_0s_1\tau))\cong Y(s_0s_1\tau),\quad \pi(Y(s_0\tau))\cong Y(s_0\tau),
\end{align*}
and
\begin{align*}
X_\mu(\tau)=&\bigsqcup_{j\in \J/\J\cap P_{\{s_0,s_1,s_3,s_4\}}}j\pi(Y(s_0s_1s_2s_0s_1s_0\tau))\sqcup \bigsqcup_{j\in \J/\J\cap P_{\{s_1,s_2,s_3\}}}j\pi(Y(s_0s_1s_2s_0\tau))\\
\sqcup &\bigsqcup_{j\in\J/\J \cap P_{\{s_0,s_1,s_3,s_4\}}}j\pi(Y(s_0s_1s_0\tau))\sqcup \bigsqcup_{j\in\J/\J \cap P_{\{s_0,s_1,s_3,s_4\}}}j\pi(Y(s_0s_1\tau))\\
\sqcup &\bigsqcup_{j\in\J/\J \cap P_{\{s_0,s_2,s_4\}}}j\pi(Y(s_0\tau))\sqcup \bigsqcup_{j\in\J/\J \cap P_{\{s_1,s_2,s_3,s_4\}}}j\{pt\}.
\end{align*}
\end{theo}
\begin{proof}
By Lemma \ref{genus 4 empty}, we have
\begin{align*}
X_\mu(\tau)&=\bigsqcup_{w\in \SAdm(\mu)_0}\pi(X_w(\tau))\\
&=\pi(X_{s_0s_1s_2s_0s_1s_0\tau}(\tau))\sqcup \pi(X_{s_0s_1s_2s_0\tau}(\tau))\sqcup \pi(X_{s_0s_1s_0\tau}(\tau))\\
&\quad \sqcup \pi(X_{s_0s_1\tau}(\tau))\sqcup \pi(X_{s_0\tau}(\tau))\sqcup \pi(X_{\tau}(\tau)).
\end{align*}
Note that $S_{s_0s_1\tau}=\emptyset$, $S_{s_0\tau}=\{s_2\}$ and $S_{\tau}=\{s_1,s_2,s_3\}$.
By Proposition \ref{spherical}, we also have 
\begin{align*}
\pi(Y(s_0s_1\tau))\cong Y(s_0s_1\tau),\quad \pi(Y(s_0\tau))\cong Y(s_0\tau),
\end{align*}
and
\begin{align*}
&\pi(X_{s_0s_1\tau}(\tau))=\bigsqcup_{j\in \J/\J\cap P_{\{s_0,s_1,s_3,s_4\}}}j\pi(Y(s_0s_1\tau)),\\
&\pi(X_{s_0\tau}(\tau))=\bigsqcup_{j\in \J/\J\cap P_{\{s_0,s_2,s_4\}}}j\pi(Y(s_0\tau)),\quad \pi(X_{\tau}(\tau))=\bigsqcup_{j\in\J/\J \cap P_{\{s_1,s_2,s_3\}}}j\{pt\}.
\end{align*}

By Lemma \ref{closed}, $X_{w}(\tau)$ is closed in $\pi^{-1}(\pi(X_{w}(\tau)))$ for $w\in \SAdm(\mu)_0$.
By \cite[Lemma 5.4]{HNY22}, the map $X_{w}(\tau)\rightarrow \pi(X_{w}(\tau))$ induced by $\pi$ is injective if $w=s_0s_1s_2s_0s_1s_0\tau$.
If $w=s_0s_1s_2s_0\tau$ (resp.\ $s_0s_1s_0\tau$), let $x\in W_0$ such that $xw=wx$.
Then $x\mu=\mu$ and $xp(w)x^{-1}=p(w)$, i.e., $x\in W_{\{s_1,s_2,s_3\}}$ and $x(1\ 5\ 4\ 2)(3\ 8\ 7\ 6)x^{-1}=(1\ 5\ 4\ 2)(3\ 8\ 7\ 6)$ (resp.\ $x(1\ 5\ 4\ 2\ 8\ 7\ 3)x^{-1}=(1\ 5\ 4\ 2\ 8\ 7\ 3)$).
It easily follows from these conditions that $x=1$.
Thus, by \cite[Lemma 2.1]{Shimada4}, $X_{w}(\tau)\rightarrow \pi(X_{w}(\tau))$ induced by $\pi$ is injective.
Since $\pi$ is proper, it induces a universal homeomorphism between $X_w(\tau)$ and $\pi(X_w(\tau))$ for $w=s_0s_1s_2s_0s_1s_0\tau$, $s_0s_1s_2s_0\tau$ or $s_0s_1s_0\tau$.
Therefore the theorem follows from Lemma \ref{genus 4 Iwahori}.
\end{proof}

\begin{rema}
\label{not positive}
In the case of genus $4$, $(G,\mu)$ is not of positive Coxeter type.
Indeed, $s_0s_1s_0\tau\xrightarrow{s_0} s_1s_0s_4\tau \xrightarrow{s_4} s_1\tau$.
By Theorem \ref{empty}, both $X_{s_1\tau}(\tau)$ and $X_{s_1s_0\tau}(\tau)$ are non-empty.
Thus $s_0s_1s_0\tau$ does not have positive Coxeter part by Theorem \ref{simple} (or \cite[Theorem A]{SSY23}, which is a theorem on reduction trees).
\end{rema}

\section{The Case of Genus $\geq 5$}
\label{genus 5}
We put $n \geq 5$, $G = \GSp_{2n}$, and $\mu =(1^{(n)}, 0^{(n)}) \in X_{\ast} (T)$.
Let $W_{0}$ be the subgroup of the symmetric group $S_{2n}$ generated by 
$s_{i} = (i\ i+1) (2n-i\ 2n+1-i)$ for $1 \leq i \leq n$, and $W_{a}$ the subgroup generated by $s_{1}, \ldots, s_{n}$ and the simple affine reflection
 \[
 s_{0} = \varpi^{(1^{(1)}, 0^{(2n-2)}, (-1)^{(1)})} (1,2n) \in W_{a},
 \]
  We put $\tau := \varpi^{\mu}(s_{n} s_{n-1} \cdots s_{1}) (s_{n} s_{n-1} \cdots s_{2}) \cdots (s_{n}s_{n-1}) (s_{n})\in \Omega$.
  We consider the element
  \[
w:= s_0 s_1 s_2 \cdots s_{n-2} s_0 s_1 s_0 \tau \in \tw.
  \]
  Then it is straightforward to check that
  \[
   w \in \SAdm (\mu).
  \]
  Moreover, we have
  \[
w \xrightarrow{s_0} s_1 s_2 \cdots s_{n-2} s_0 s_1 s_0 s_n \tau
\xrightarrow{s_{n}}
s_1 s_2 \cdots s_{n-2} s_0 s_1 \tau
  \]
  Therefore, by Proposition \ref{DL method prop}, there exists a (possibly empty) locus of $X_{w} (\tau)$ which admits a morphism onto $X_{s_{1} s_{2} \cdots s_{n-2} s_{0} s_{1} s_{0} \tau} (\tau)$ which is a composition of a universal homeomorphism and a $\mathbb{G}_{m}$-bundle.
  Furthermore, we have
  \begin{eqnarray*}
&& s_{1} s_{2} \cdots s_{n-2} s_{0} s_{1} s_{0} \tau \\
  &\xrightarrow{s_2} &
  s_2 s_1 s_2 \cdots s_{n-3} s_0 s_1 s_0 \tau \\
  &\xrightarrow{s_3}&
  s_3 s_2 s_1 s_2 \cdots s_{n-3} s_0 s_1 s_0 \tau  \\
  &\cdots& \\
  &\xrightarrow{s_{n-3}}&
  s_{n-3} s_{n-4} \cdots s_{1} s_{2}
s_{0} s_{1} s_{0} \tau \\
&=& s_{n-3} s_{n-4} \cdots s_{3} s_{1} s_{2} s_{1} s_{0} s_{1} s_{0} \tau\\
&=&  s_{n-3} s_{n-4} \cdots s_{3} s_{1} s_{2} s_{0} s_{1} s_{0} s_{1} \tau \\
&\xrightarrow{s_{1}}& s_{n-3} s_{n-4} \cdots s_{3} s_{2} s_{0} s_{1} s_{0} s_{1} s_{n-1} \tau\\
&\xrightarrow{s_{n-1}}& s_{n-3} s_{n-4} \cdots s_{3} s_{2} s_{0} s_{1} s_{0} \tau\\
&\xrightarrow{s_{0}}& s_{n-3} s_{n-4} \cdots s_{3} s_{2} s_{1} s_{0} s_{n} \tau \\
&\xrightarrow{s_{n}}& s_{n-3} s_{n-4} \cdots s_{3} s_{2} s_{1} \tau,
  \end{eqnarray*}
  and by Proposition \ref{DL method prop} again, there exists a locus of $X_{s_{1} s_{2} \cdots s_{n-2} s_{0} s_{1} s_{0} \tau} (\tau)$ which admits a composition of universal homeomorphisms and $\mathbb{A}^{1}$-bundles onto $X_{s_{2} s_{1} \tau} (\tau)$.
   Since we have 
   \[
   \supp_{\sigma}(s_{n-3} s_{n-4} \cdots s_{3} s_{2} s_{1} \tau,) = (s_{1}, s_{2}, s_{3}, s_{n-1}) \neq \tS,
   \]   
   $X_{s_{n-3} s_{n-4} \cdots s_{3} s_{2} s_{1} \tau,} (\tau)$ is non-empty.
   This example shows that for $n \geq 5$, there exists a $w \in \SAdm (\mu)$ such that some reduction tree of $X_{w} (\tau)$ contains a non-trivial $\G_{m}$-fibration. That means, the analogue of Theorem \ref{main theo2} does not hold in $n\geq 5.$
   In particular, for any $n \geq 4$, $(G,\mu)$ is not of positive Coxeter type (cf.\ Remark \ref{not positive} and Theorem \ref{simple}).

\bibliographystyle{myamsplain}
\bibliography{reference}
\end{document}